\documentclass[11pt,reqno]{amsart}
\usepackage{amsmath,amssymb,amsthm,mathtools}
\usepackage{booktabs,tabularx}
\usepackage[margin=1in]{geometry}
\usepackage{xcolor}
\usepackage[colorlinks=true,linkcolor=blue!50!black,citecolor=blue!50!black,urlcolor=blue!50!black]{hyperref}

\hypersetup{
  pdftitle={Carlsson's Conjecture and the Generalized Total Rank Conjecture in Characteristic Two},
  pdfauthor={Keller VandeBogert},
  pdfsubject={2020 Mathematics Subject Classification: Primary 13D02; Secondary 13A35,
    13D22, 20C20, 55M35},
  pdfkeywords={generalized total rank conjecture, Carlsson's conjecture, differential modules,
    projective flags, finite group actions, Tate construction, Frobenius pullback}
}

\theoremstyle{plain}
\newtheorem{theorem}{Theorem}[section]
\newtheorem{proposition}[theorem]{Proposition}
\newtheorem{lemma}[theorem]{Lemma}

\newtheorem{thmintro}{Theorem}

\newtheorem{corintro}[thmintro]{Corollary}
\theoremstyle{definition}
\newtheorem{remark}[theorem]{Remark}
\newtheorem{example}[theorem]{Example}

\DeclareMathOperator{\Tor}{Tor}
\DeclareMathOperator{\length}{length}
\DeclareMathOperator{\rank}{rank}
\DeclareMathOperator{\im}{im}
\DeclareMathOperator{\id}{id}
\DeclareMathOperator{\codim}{codim}
\DeclareMathOperator{\height}{ht}
\DeclareMathOperator{\Supp}{Supp}
\newcommand{\Ff}{\mathbb{F}}
\newcommand{\Z}{\mathbb{Z}}
\newcommand{\Tate}{T}                 
\newcommand{\Frob}{F_R^{*}}           
\newcommand{\Rone}{R_F}               
\newcommand{\htot}{h}

\title[Carlsson's conjecture and generalized total ranks]
      {Carlsson's Conjecture and the Generalized Total Rank Conjecture in Characteristic Two}
\author{Keller VandeBogert}
\address{Department of Mathematics, University of Kentucky, 719 Patterson Office Tower,
Lexington, KY 40506, USA}
\email{keller.v@uky.edu}
\date{}

\begin{document}

\begin{abstract}
We prove the generalized total rank conjecture over regular rings in characteristic $2$: if $R$ is
 a regular Noetherian
 domain of characteristic $2$ and $P$ is a differential $R$-module admitting a finite projective
 flag and having nonzero homology $H(P)$, then $\rank_R(P)\ge2^{\codim_RH(P)}$.  In particular, we
 prove Carlsson's conjecture for elementary abelian $2$-groups in every rank.  We also
 obtain sharp homology bounds for arbitrary continuous actions of such groups and for perfect complexes
 over finite group algebras; the sphere rank conjecture follows.  The proof identifies the homology
 of a Tate construction on the tensor square $P\otimes_RP$ with the Frobenius pullback of $H(P)$,
 and compares lengths by deforming the Tate
 differential.
\end{abstract}

\maketitle

\section{Introduction}\label{sec:intro}

Conjectures about total ranks of free complexes and about the homology of spaces with free group
actions lie at the intersection of commutative algebra and algebraic topology.  On the topological
side, Carlsson~\cite{CarlssonHomology} conjectured~\cite[Conjecture~I.3]{Carlsson} that a
nonempty finite free $(\Z/2)^d$-CW complex $X$
satisfies $\sum_i\dim_{\Ff_2}H_i(X;\Ff_2)\ge 2^d$, and proved it for
$d\le3$~\cite[Theorem~2]{CarlssonStudies}.  On the algebraic side, the Buchsbaum--Eisenbud--Horrocks
conjecture~\cite{BuchsbaumEisenbud,Hartshorne} predicts that a nonzero finite-length module of finite
projective dimension over a Noetherian local ring of dimension $d$ has $i$th Betti number at least
$\binom{d}{i}$.  Summing these inequalities gives the total rank conjecture, proposed by Avramov and
formulated and studied by Avramov--Buchweitz~\cite{AvramovBuchweitz}: the total Betti number is at
least $2^d$.  The corresponding
form for complexes asserts that every bounded complex $P$ of finite free modules with nonzero
finite-length homology over a local ring $R$ of dimension $d$ satisfies
$\rank_R(P):=\sum_i\rank_R P_i\ge 2^d$.

This extension to complexes, known as the \emph{generalized total rank conjecture}, is false in odd
characteristic: Iyengar and Walker~\cite{IyengarWalker} produced counterexamples, even over regular
local rings, and refuted the algebraic analogue of Carlsson's conjecture for $(\Z/p)^d$ with $p$
odd; these complexes are not topologically realizable~\cite{RupingStephan}.  In contrast,
Walker~\cite{Walker} proved the total rank conjecture for modules in odd characteristic.  Neither
the counterexamples nor Walker's theorem addressed characteristic $2$.  Evidence that characteristic
$2$ is exceptional came from~\cite{VW}, which proved a statement about sufficiently short complexes
known to be false in every odd characteristic.  Our main theorem proves the generalized conjecture
over regular rings of characteristic $2$, and Carlsson's conjecture in full.

\begin{thmintro}\label{cor:TRC}
Let $R$ be a regular Noetherian domain of characteristic $2$.  If $P$ is a bounded complex of
finitely generated free $R$-modules or, more generally, a differential $R$-module admitting a
finite projective flag, and $H(P)\ne0$, then $\rank_R(P)\ge 2^{\codim_R H(P)}$.
\end{thmintro}

For a group $G$ of finite elementary abelian $2$-rank, write
$r_2(G)=\max\{e:(\Z/2)^e\le G\}$.  We~deduce Carlsson's conjecture and the following sharp bound
for arbitrary continuous actions.

\begin{corintro}\label{cor:orbits}
Let $E=(\Z/2)^d$ act continuously on a nonempty finite CW complex $X$.  If the action is free, then
\[
 \sum_i\dim_{\Ff_2}H_i(X;\Ff_2)\ge2^d.
\]
Thus Carlsson's conjecture holds in every rank.  More generally, without assuming
freeness, set $s=\max_{x\in X}\rank E_x$, where $E_x$ is the stabilizer of $x$.  Then
\[
 \sum_i\dim_{\Ff_2}H_i(X;\Ff_2)\ge2^{d-s}=\min_{x\in X}|E\cdot x|.
\]
Finally, if a topological group $G$ of finite elementary abelian $2$-rank acts continuously on $X$, then
\[
 \sum_i\dim_{\Ff_2}H_i(X;\Ff_2)\ge2^{r_2(G)-\max_{x\in X}r_2(G_x)}.
\]
\end{corintro}

\begin{corintro}\label{cor:Carlsson}
Let $k$ be a field of characteristic $2$ and let $G$ be a finite group.  Every bounded complex $C$ of
finitely generated projective $kG$-modules with nonzero homology satisfies
\[
 \sum_i\dim_k H_i(C)\ge2^{r_2(G)}.
\]
This bound is sharp.  Consequently, every nonempty finite CW complex $X$ with a
free continuous $G$-action satisfies
\[
 \sum_i\dim_{\Ff_2}H_i(X;\Ff_2)\ge2^{r_2(G)}.
\]
\end{corintro}

The sphere rank problem asks whether a free action of $E=(\Z/2)^d$ on a product of $m$ spheres
forces $d\le m$.  It has roots in the results of Smith, Conner, and Heller for one and two
factors~\cite{Smith,Conner,Heller} and was the original topological motivation for Carlsson's
conjecture~\cite{CarlssonSpheres,CarlssonSphereRank}; see
also~\cite{AdemBrowder,Hanke,OkutanYalcin}.
\begin{corintro}\label{cor:sphere-rank}
Let $E=(\Z/2)^d$ act freely and continuously on a finite CW complex
$X\simeq\prod_{j=1}^m S^{n_j}$, where $n_j\ge1$.  Then $d\le m$.
\end{corintro}

Indeed, the K\"unneth theorem gives
\[
 \sum_i\dim_{\Ff_2}H_i(X;\Ff_2)=2^m,
\]
so Corollary~\ref{cor:orbits} yields $2^d\le2^m$, and hence $d\le m$.  For an arbitrary continuous $E$-action
on such an $X$, Corollary~\ref{cor:orbits} still gives a point whose stabilizer has rank at least
$d-m$.  For comparison, if a finite CW complex
$X\simeq\prod_{j=1}^m\mathbb{R}P^{n_j}$ admits a free $E$-action, Corollary~\ref{cor:orbits} gives
$2^d\le\prod_{j=1}^m(n_j+1)$; Yal\c{c}{\i}n recently obtained sharper bounds in terms of the dimensions for
free cellular actions on finite-dimensional CW complexes of this homotopy type, under the additional
assumption that the action on mod-$2$ cohomology is trivial~\cite[Theorem~2]{YalcinRP}.

The formulation of the generalized total rank conjecture for differential modules is due to
Avramov, Buchweitz, and Iyengar~\cite[Conjecture~5.3]{ABI}; Theorem~\ref{cor:TRC} settles the regular local case in
characteristic $2$.  The specialization to complexes follows by replacing $P_\bullet$ with the
differential module $\bigoplus_iP_i$, equipped with the differential induced by that of $P_\bullet$;
it has a finite projective flag, rank $\sum_i\rank_RP_i$, and homology $\bigoplus_iH_i(P_\bullet)$.  Over a local ring the
projective modules $P_i$ are free, so
Theorem~\ref{cor:TRC} contains the generalized total rank conjecture for complexes and, via Koszul
duality~\cite[Theorem~II.7 and Proposition~II.9]{Carlsson}, yields Carlsson's conjecture, the free case of
Corollary~\ref{cor:orbits}.

In joint work in preparation with Mark E. Walker, we give a direct topological proof of
Carlsson's conjecture in the finite simplicial setting.  More generally, if $p$ is prime and $X$
is a nonempty finite simplicial set equipped with a simplicial action of $E=(\Z/p)^r$ that is free
in every degree, then
\[
 \left(\sum_i\dim_{\Ff_p}H_i(X;\Ff_p)\right)^{p-1}\ge p^r.
\]

\begin{samepage}
For orientation, Table~\ref{tab:rank-status} places Theorem~\ref{cor:TRC} and
Corollaries~\ref{cor:orbits}, \ref{cor:Carlsson}, and~\ref{cor:sphere-rank} alongside the
corresponding earlier results.  Together with the counterexamples of Iyengar and Walker~\cite{IyengarWalker},
Theorem~\ref{cor:TRC} establishes the dichotomy between characteristic $2$ and the other
characteristics for the generalized total rank conjecture over equicharacteristic regular local
rings.  Theorem~\ref{cor:TRC} is sharp for Koszul complexes; the bounds in
Corollary~\ref{cor:orbits} are sharp for the translation actions on $E$ and $E/H$; and
Corollary~\ref{cor:Carlsson} is sharp by the Benson--Carlson complexes constructed in its proof.\par
\end{samepage}

\begin{table}[t]
\caption{Consequences of Theorem~\ref{cor:TRC} for rank problems.}
\label{tab:rank-status}
\centering
\footnotesize
\setlength{\tabcolsep}{4pt}
\renewcommand{\arraystretch}{1.14}
\begin{tabularx}{\textwidth}{
 @{}>{\raggedright\arraybackslash}p{.25\textwidth}
 >{\raggedright\arraybackslash}X
 >{\raggedright\arraybackslash}X@{}}
\toprule
Problem & Earlier results & Consequence of Theorem~\ref{cor:TRC} \\
\midrule
Generalized total rank
& Open in characteristic $2$ for arbitrary bounded complexes and differential modules admitting
finite projective flags;
counterexamples exist in residual characteristic different from $2$ for $d\ge8$~\cite{IyengarWalker}.
& The codimension bound of Theorem~\ref{cor:TRC} over regular rings of characteristic $2$. \\
\addlinespace[2pt]
Perfect complexes over group algebras
& Known for $k[(\Z/2)^d]$ when $d\le3$~\cite[Theorem~2]{CarlssonStudies} and open for $d\ge4$;
false at odd primes for $d\ge8$~\cite{IyengarWalker}.
& The minimum total homology dimension is $2^{r_2(G)}$ among perfect $kG$-complexes with nonzero
homology, for every finite group $G$ and every field $k$ of characteristic $2$
(Corollary~\ref{cor:Carlsson}). \\
\addlinespace[2pt]
Actions of finite groups
& Carlsson's conjecture was known for free $(\Z/2)^d$-actions with $d\le3$
\cite[Theorem~2]{CarlssonStudies}.
& Carlsson's conjecture in every rank (Corollary~\ref{cor:orbits}); more generally, for continuous
actions on a nonempty finite CW complex,
$\sum_i\dim_{\Ff_2}H_i(X;\Ff_2)\ge2^{r_2(G)}$ when a finite group $G$ acts freely, and
$\sum_i\dim_{\Ff_2}H_i(X;\Ff_2)\ge\min_{x\in X}|E\cdot x|$ for every $E$-action. \\
\addlinespace[2pt]
Sphere rank
& No proof was known without restrictions on the dimensions or cohomology action; see
\cite{CarlssonSpheres,CarlssonSphereRank,AdemBrowder,Hanke,OkutanYalcin}.
& A free continuous $(\Z/2)^d$-action on a finite CW complex
$X\simeq\prod_{j=1}^mS^{n_j}$ forces $d\le m$, without restrictions on the dimensions or
cohomology action. \\
\bottomrule
\end{tabularx}
\end{table}

These results rest on the following computation.  Let $(R,\mathfrak m)$ be a commutative Noetherian
local ring of characteristic $2$, and let $P$ be a differential $R$-module.  Let $\Rone$ denote
$R$ as a bimodule with regular left action and
right action twisted by Frobenius, $x\cdot r=xr^2$, and let $\Frob P=\Rone\otimes_R P$ be the
\emph{Frobenius pullback}.  Put $\partial=d_P\otimes1+1\otimes d_P$, and let $\tau$ be the
transposition $x\otimes y\mapsto y\otimes x$ on $P\otimes_R P$.  The \emph{Tate construction} on $P$
is the differential module $\Tate_R(P)=\bigl(P\otimes_RP,\ \partial+(1+\tau)\bigr)$
(Section~\ref{sec:deduce}).

\begin{thmintro}\label{thm:main}
Let $R$ be a regular local ring of characteristic $2$, and let $P$ be a differential $R$-module
admitting a finite free flag.  Then $H(\Tate_R(P))\cong\Rone\otimes_RH(P)$; equivalently
$H(\Tate_R(P))\cong H(\Frob P)$, the Frobenius of $R$ being flat.
\end{thmintro}

The appearance of Frobenius is already visible when $P$ is free with zero differential.  Relative
to a basis of $P$, each submodule spanned by $b\otimes c$ and $c\otimes b$, with $b\ne c$, is
stable under $1+\tau$ and acyclic with that differential, while the classes of the diagonal tensors
$b\otimes b$ form the homology.
Modulo the image of $1+\tau$, the assignment $x\mapsto x\otimes x$ is additive, and
$(rx)\otimes(rx)=r^2(x\otimes x)$, so these diagonal classes carry the Frobenius twist.
Section~\ref{sec:hypertor} extends this calculation to the filtered differential module constructed in
Lemma~\ref{lem:filtration}.

For a differential module $C$ with finite-length homology write $\htot(C)=\length_RH(C)$.  Once
Theorem~\ref{thm:main} is known, the rank bound follows by counting lengths: over a regular local ring
of dimension $d$ the Frobenius pullback multiplies finite length by $2^d$, so
\[
 2^d\,\htot(P)\ =\ \htot(\Rone\otimes_RH(P))\ =\ \htot(\Tate_R(P))\ \le\ \htot(P\otimes_R P)\ \le\
 \rank_R(P)\,\htot(P).
\]
The equalities use the Frobenius length formula and Theorem~\ref{thm:main}; the first inequality is
proved by deforming the differential of the Tate construction, and the second by filtering along a
flag of one tensor factor.
Since $\htot(P)>0$, cancellation gives $\rank_R(P)\ge 2^d$.  The general statement of
Theorem~\ref{cor:TRC} follows by localizing at a prime of minimal height in $\Supp H(P)$, where the
homology becomes finite length and Theorem~\ref{thm:rank} applies; that height is
$\codim_R H(P)$.

The proof of Theorem~\ref{thm:main} uses an Adams spectral sequence.  We replace $P$ by a
quasi-isomorphic differential module with a finite filtration whose subquotients are the terms
$F_q$ of a finite free resolution $F_\bullet\to H(P)$ and whose differential induces the resolution
maps $F_q\to F_{q-1}$.  Exactness of $\Tate_R$ then gives a finite exact couple and hence a
spectral sequence.
Its first page is the Tate homology of the subquotients, which Lemma~\ref{lem:zerodiff} identifies
with their Frobenius pullbacks, so its second page is $\Tor^R_*(\Rone,H(P))$.  Regularity makes the
Frobenius flat by Kunz's theorem, and the sequence collapses.

The bound $\htot(\Tate_R(P))\le\htot(P\otimes_R P)$ is a separate length comparison.  Over
$R[s]_{(\mathfrak m,s)}$ the differential $\partial+s(1+\tau)$ interpolates between $\partial$ and
the Tate differential: the special fiber of the deformed homology embeds into $H(P\otimes_RP)$, the
generic fiber is the Tate construction of the differential module $(P,s^{-1}d_P)$, and a
Hilbert--Samuel multiplicity argument gives the stated inequality.

The argument takes place entirely in the category of differential modules admitting finite free
flags: no grading is assumed, bounded complexes give such modules by taking the direct sum of their
terms, and over
$S=k[x_1,\dots,x_d]$ with $|x_i|=-1$, the polynomial ring appearing in Carlsson's algebraic
formulation, every finite free dg module admits a free flag (Lemma~\ref{lem:flag}).

The Frobenius twist in the $C_2$-Tate construction of a tensor square ($C_2$ cyclic of order two) is
a feature of the Tate diagonal of Nikolaus and Scholze~\cite{NS} and Carmeli's Frobenius twist
functor~\cite[Definition~2.9]{Carmeli}.  Frobenius twists also occur in Kaledin's work on cyclic
homology in positive characteristic~\cite{Kaledin}, and Lipshitz and Treumann use a Tate
construction on chain complexes in characteristic $2$~\cite{LipshitzTreumann}.
Our construction is self-contained and carried out entirely on chains.  It applies to any differential module
with a finite free flag over a ring of characteristic $2$; over a regular local ring, Kunz's theorem
supplies the flatness that collapses the Adams spectral sequence to Theorem~\ref{thm:main}.

\medskip\noindent\textbf{Acknowledgments.}\ I am deeply grateful to Mark E.~Walker for a careful reading
of the manuscript and for many valuable comments that significantly improved the original arguments.
I thank Benjamin Briggs for valuable comments on an earlier draft that led to the projective-flag
formulation of Theorem~\ref{cor:TRC}.

\section{The Tate construction and the rank bound}\label{sec:deduce}

Throughout this section $A$ is a commutative Noetherian ring with $2=0$, and all tensor products are
over $A$.  A \emph{differential $A$-module} is an $A$-module $P$ with an $A$-linear endomorphism
$d_P$ satisfying $d_P^{2}=0$.  Its homology is $H(P)=\ker d_P/\im d_P$.  A morphism commutes with
the differentials, and morphisms $f,g\colon P\to Q$ are homotopic if
$f+g=d_Q\sigma+\sigma d_P$ for some $A$-linear map $\sigma\colon P\to Q$.  A short exact sequence
$0\to X\to Y\to Z\to0$ of differential modules induces the exact homology sequence
\[
\cdots\longrightarrow H(X)\longrightarrow H(Y)\longrightarrow H(Z)\longrightarrow
H(X)\longrightarrow\cdots.
\]

A differential module is \emph{finite projective} (respectively, \emph{finite free}) if its
underlying module is finitely generated projective (respectively, finite free).  A \emph{free flag}
on $P$ is a finite filtration
\[
0=P^{0}\subset P^{1}\subset\cdots\subset P^{r}=P
\]
by direct summands whose subquotients are finite free and which the differential \emph{drops},
$d_P(P^{i})\subseteq P^{i-1}$; then $P$ is finite free, and $\rank_A(P)$ denotes its rank.  A
\emph{projective flag} is defined in the same way, with finitely generated projective subquotients;
it makes $P$ finite projective, with locally constant rank.  Moreover, a projective flag exhibits
$P$ as an object of the thick subcategory of the homotopy category generated by $A$ with zero
differential; thus $P$ is perfect in the usual triangulated sense.  Differential modules, free
flags, and projective flags were studied by Avramov, Buchweitz, and Iyengar~\cite{ABI}.  For a bounded complex
of finitely generated projective modules, the direct-sum differential module $\bigoplus_nP_n$ has a
projective flag by homological degree; when the terms are free, it has a free flag;
a flag requires no grading, and the differential need not drop it by exactly one step.  When $H(C)$ has finite
length we write $\htot(C)=\length_AH(C)$.

Finally, $A_F$ denotes $A$ as an $A$-bimodule with the regular left action
and the right action twisted by Frobenius, $x\cdot a=xa^{2}$, and $F_A^{*}(-)=A_F\otimes_A(-)$; when $A=R$,
this agrees with the Frobenius pullback of the introduction, and if the differential of $P$ is represented in a
basis by a matrix $(a_{\lambda\mu})$, that of $F_A^{*}P$ is represented by $(a_{\lambda\mu}^2)$.

The following homotopical consequence of a projective flag will be used repeatedly.  The flag hypothesis
is essential: the differential module $(A,\epsilon\cdot)$ over
$A=k[\epsilon]/(\epsilon^2)$ is acyclic but not contractible.

\begin{lemma}\label{lem:contract}
An acyclic differential $A$-module with a projective flag is contractible, and every
quasi-isomorphism between differential $A$-modules with projective flags is a homotopy
equivalence.
\end{lemma}

\begin{proof}
The first assertion is a special case of a theorem of Avramov, Buchweitz, and
Iyengar~\cite[Theorem~2.3]{ABI}: an acyclic retract of a differential module admitting a projective
flag is contractible.  For the second assertion, let $u\colon P\to Q$ be a morphism between
differential modules with projective flags.  Its mapping cone is $Q\oplus P$ with differential
$(y,x)\mapsto(d_Qy+ux,\,d_Px)$.  First taking the flag of $Q$ and then the inverse images of the
flag of $P$ under $Q\oplus P\to P$ gives the cone a projective flag.  If $u$ is a quasi-isomorphism, the
exact homology sequence for $0\to Q\to Q\oplus P\to P\to0$, whose connecting map is $H(u)$, shows
that the cone is acyclic.  A contracting homotopy for the cone then yields a homotopy inverse to
$u$.
\end{proof}

The group $C_2=\{1,\tau\}$ acts on $P\otimes P$ by the transposition $\tau(x\otimes y)=y\otimes x$,
which commutes with $\partial=d_P\otimes1+1\otimes d_P$.  Set
\begin{equation}\label{eq:tate}
 \Tate_A(P)=\bigl(P\otimes P,\ \partial+(1+\tau)\bigr),
\end{equation}
the \emph{Tate construction} on $P$.  For a morphism $f\colon P\to Q$, set
$\Tate_A(f)=f\otimes f$.  This is a differential module, since
\[
(\partial+(1+\tau))^2
 =2(d_P\otimes d_P)+2(1+\tau)\partial+2(1+\tau)=0.
\]

\begin{remark}\label{rem:chartwo}
Differential modules and free flags make sense over any ring~\cite{ABI}; characteristic $2$ is
exactly what makes the usual tensor-product differential work without a grading.  With no Koszul
signs, this endomorphism on $P\otimes Q$ is $d_P\otimes1+1\otimes d_Q$, whose square is
$2(d_P\otimes d_Q)$; it is a differential for all $P,Q$ exactly when $2=0$.  This gives a symmetric
monoidal structure, with symmetry the transposition $\tau$ in~\eqref{eq:tate}, and differential
modules admitting free flags are closed under tensor products, by
$(P\otimes Q)^{k}=\sum_{i+j=k}P^{i}\otimes Q^{j}$.
\end{remark}

\begin{lemma}\label{lem:htpy}
The construction $\Tate_A$ preserves homotopies and is additive up to
homotopy, so it descends to the homotopy category of finite projective differential $A$-modules.  It is
\emph{exact}: every short exact sequence $0\to X\to Y\to Z\to0$ of finite projective differential $A$-modules
induces the exact homology sequence
\[
\cdots\to H\bigl(\Tate_A(X)\bigr)\to H\bigl(\Tate_A(Y)\bigr)\to H\bigl(\Tate_A(Z)\bigr)\to
H\bigl(\Tate_A(X)\bigr)\to\cdots.
\]
\end{lemma}

\begin{proof}
Write $D_P$ and $D_Q$ for the Tate differentials
on $P\otimes P$ and $Q\otimes Q$, respectively.  For homotopy invariance, if
$f+g=d_Q\sigma+\sigma d_P$, then $K=\sigma\otimes f+g\otimes\sigma+(\sigma\otimes\sigma)\tau$ satisfies
$D_QK+KD_P=f^{\otimes2}+g^{\otimes2}$; hence $f\simeq g$ implies
$\Tate_A(f)\simeq\Tate_A(g)$.  For additivity,
$\Tate_A(f+g)=\Tate_A(f)+\Tate_A(g)+(f\otimes g+g\otimes f)$, and the cross term is
null-homotopic via $(f\otimes g)\tau$.

For exactness, filter $Y\otimes Y$ by the $C_2$-stable differential submodules
\[
F_0=X\otimes X,\qquad F_1=(X\otimes Y)+(Y\otimes X),\qquad F_2=Y\otimes Y,
\]
with associated graded $X\otimes X$, $(X\otimes Z)\oplus(Z\otimes X)$, and $Z\otimes Z$.
Since $Z$ is projective, the underlying sequence splits, so this filtration is split as modules.
The outer subquotients are $\Tate_A(X)$ and
$\Tate_A(Z)$.  The middle subquotient is $M\oplus M$, where $M=X\otimes Z$, with $\tau$
swapping the two summands and differential $D=d_M\oplus d_M+(1+\tau)$.  The operator
$h(m_1,m_2)=(m_2,0)$ satisfies $Dh+hD=\id$, since $(1+\tau)h+h(1+\tau)=\id$ in characteristic $2$ and
$h$ commutes with $d_M\oplus d_M$.  Thus the middle subquotient is contractible,
$\Tate_A(Y)/\Tate_A(X)\to\Tate_A(Z)$ is a quasi-isomorphism, and the short exact sequence
$0\to\Tate_A(X)\to\Tate_A(Y)\to\Tate_A(Y)/\Tate_A(X)\to0$ yields the asserted exact homology
sequence.
\end{proof}

\begin{remark}
Lemma~\ref{lem:htpy} is the analogue, for chain complexes, of the exactness of the Frobenius twist functor
$X\mapsto(X^{\otimes2})^{tC_2}$ in stable homotopy theory; see
Nikolaus--Scholze~\cite[Proposition~III.1.1]{NS},
Carmeli~\cite[Proposition~2.10]{Carmeli}, and, for the Tate
construction on chains, Lawson~\cite[Proposition~2.13]{Lawson}.
\end{remark}

The following proposition gives the length comparison used in Theorem~\ref{thm:rank}.

\begin{proposition}\label{prop:size}
Let $(A,\mathfrak m)$ be local and let $P$ be a finite projective differential $A$-module with
$\htot(P\otimes P)<\infty$.  Set $B=A[s]_{(\mathfrak m,s)}$ and $\mathfrak p=\mathfrak mB$, and let
\[
M=H\bigl((P\otimes P)\otimes_AB,\ \partial+s(1+\tau)\bigr).
\]
Then $\length_{B_{\mathfrak p}}(M_{\mathfrak p})\le\htot(P\otimes P)$.
\end{proposition}

\begin{proof}
Write $\mathcal C=((P\otimes P)\otimes_AB,\partial+s(1+\tau))$, so $M=H(\mathcal C)$ is a
finitely generated $B$-module.  Multiplication by $s$ is injective on $\mathcal C$, since $s$ is a
nonzerodivisor on $B$ and $\mathcal C$ is finite free; its cokernel is $(P\otimes P,\partial)$,
since $B/sB=A$ and the term $s(1+\tau)$ vanishes.  The exact homology sequence for
$0\to\mathcal C\xrightarrow{s}\mathcal C\to(P\otimes P,\partial)\to0$ therefore embeds $M/sM$ into
$H(P\otimes P)$, so
\[
\length_B(M/sM)\ \le\ \htot(P\otimes P)\ <\ \infty.
\]
Since $M/sM$ has finite length, $\dim M\le\dim(M/sM)+1\le1$.

If $\dim M\le0$ then $M_{\mathfrak p}=0$, since $\dim B/\mathfrak p=1$ shows that $\mathfrak p$ is
not the maximal ideal, and there is nothing to prove.  Otherwise $s$ is a parameter for the
one-dimensional module $M$, because $M/sM$ has finite length, and the associativity formula for
multiplicities~\cite[Corollary~4.7.8]{BrunsHerzog} (for its cycle formulation
see~\cite[Proposition~5.2.11]{Roberts}), applied over $B/\operatorname{ann}M$, gives
\[
e(s;M)=\sum_{\mathfrak q}\length_{B_{\mathfrak q}}(M_{\mathfrak q})\,e(s;B/\mathfrak q),
\]
the sum running over the primes $\mathfrak q\in\Supp M$ with $\dim B/\mathfrak q=1$.  Every summand is
nonnegative, and $B/\mathfrak p\cong(A/\mathfrak m)[s]_{(s)}$, a discrete valuation ring, gives
$e(s;B/\mathfrak p)=1$.  Thus $\length_{B_{\mathfrak p}}(M_{\mathfrak p})\le e(s;M)$ if
$\mathfrak p\in\Supp M$, while the left side is zero otherwise.  Finally, for the parameter $s$ the
multiplicity agrees with
Northcott's multiplicity symbol~\cite[Definition~4.7.3, Theorems~4.7.4 and~4.7.6]{BrunsHerzog}
(see also~\cite[Theorem~5.2.8, Corollary~5.2.9]{Roberts}),
which for a single element reads
\[
e(s;M)\ =\ \length_B(M/sM)-\length_B(0:_Ms)\ \le\ \length_B(M/sM)\ \le\ \htot(P\otimes P).\qedhere
\]
\end{proof}

\begin{lemma}\label{lem:tensorbound}
Let $X$ be a differential $A$-module whose homology has finite length and let $Y$ be a differential
$A$-module with a free flag.  Then $H(X\otimes_AY)$ has finite length and
\[
\htot(X\otimes_AY)\ \le\ \rank_A(Y)\,\htot(X).
\]
\end{lemma}

\begin{proof}
For a free flag $Y^\bullet$, put $r_i=\rank_A(Y^i/Y^{i-1})$.  Tensoring this split filtration with
$X$ gives a filtration of $X\otimes_A Y$ with differential module subquotients $X^{\oplus r_i}$.
The resulting exact homology sequences show inductively that each $X\otimes_A Y^i$ has
finite-length homology and that
\[
\htot(X\otimes_A Y)\le\sum_i r_i\htot(X)=\rank_A(Y)\htot(X).\qedhere
\]
\end{proof}

\begin{lemma}\label{lem:scaling}
Let $R$ be a regular local ring of characteristic $2$ and dimension $d$.  For every finite-length
$R$-module $N$, $\length_R(\Rone\otimes_RN)=2^{d}\length_R(N)$.
\end{lemma}

\begin{proof}
By Kunz's theorem~\cite[Theorem~2.1]{Kunz} the Frobenius of $R$ is flat, so $\Rone\otimes_R-$ is exact,
and by additivity along a composition series the claim reduces to the residue field
$N=k=R/\mathfrak m$.  For a
regular system of parameters $x_1,\dots,x_d$ generating the maximal ideal $\mathfrak m$, the Frobenius
power $\mathfrak m^{[2]}=(r^2\mid r\in\mathfrak m)$ equals $(x_1^2,\dots,x_d^2)$; since
$\Rone\mathfrak m=\mathfrak m^{[2]}$, we get $\Rone\otimes_Rk=R/(x_1^2,\dots,x_d^2)$, of length $2^{d}$
because $x_1,\dots,x_d$ is a regular sequence.
\end{proof}

\begin{theorem}\label{thm:rank}
Let $R$ be a regular local ring of characteristic $2$ and dimension $d$, and let $P$ be a differential
$R$-module admitting a free flag, with $0<\htot(P)<\infty$.  Then $\rank_R(P)\ge2^{d}$.
\end{theorem}

\begin{proof}
Lemma~\ref{lem:tensorbound} gives $\htot(P\otimes P)\le\rank_R(P)\,\htot(P)<\infty$, so
Proposition~\ref{prop:size} applies; keep its notation $B=R[s]_{(\mathfrak m,s)}$,
$\mathfrak p=\mathfrak mB$, and $M$.  The local ring $B_{\mathfrak p}=R[s]_{\mathfrak mR[s]}$ is
regular of dimension $d$.  The base change $P_{\mathfrak p}=P\otimes_RB_{\mathfrak p}$ is finite free
with a free flag, and flatness gives $H(P_{\mathfrak p})=H(P)\otimes_RB_{\mathfrak p}$, of finite
length $\htot(P)$: a composition series of $H(P)$ base changes to one with factors
$B_{\mathfrak p}/\mathfrak mB_{\mathfrak p}$.  Moreover $s$ is a unit in $B_{\mathfrak p}$, and
set $Q=(P_{\mathfrak p},s^{-1}d_{P_{\mathfrak p}})$.  It has the same free flag and homology as
$P_{\mathfrak p}$.  Since
\[
 \partial_{P_{\mathfrak p}}+s(1+\tau)=s\bigl(\partial_Q+(1+\tau)\bigr),
\]
and scaling a differential by a unit changes neither its kernel nor its image, exactness of
localization gives $M_{\mathfrak p}=H(\Tate_{B_{\mathfrak p}}(Q))$.  Since $B_{\mathfrak p}$ is
regular, Theorem~\ref{thm:main} gives
$M_{\mathfrak p}\cong(B_{\mathfrak p})_F\otimes_{B_{\mathfrak p}}H(P_{\mathfrak p})$, of length
$2^{d}\,\htot(P)$ by Lemma~\ref{lem:scaling} over $B_{\mathfrak p}$.  Combining with
Proposition~\ref{prop:size} and Lemma~\ref{lem:tensorbound},
\[
2^{d}\,\htot(P)\ =\ \length_{B_{\mathfrak p}}(M_{\mathfrak p})\ \le\ \htot(P\otimes P)\ \le\
\rank_R(P)\,\htot(P),
\]
and canceling $\htot(P)>0$ yields the bound.
\end{proof}

\begin{proof}[Proof of Theorem~\ref{cor:TRC}]
Set $N=H(P)$; here $\codim_RN=\min\{\height\mathfrak q:\mathfrak q\in\Supp_RN\}$, so we may choose
$\mathfrak q\in\Supp_R N$ with $\height\mathfrak q=\codim_R N$.  Then $\mathfrak q$ is minimal in
$\Supp_R N$, and $N_\mathfrak q$ is nonzero of finite length over the regular local ring
$R_\mathfrak q$ of dimension $\height\mathfrak q$.  The localized module $P_\mathfrak q$ is a
differential $R_\mathfrak q$-module with a free flag, finitely generated projectives over a local
ring being free, so Theorem~\ref{thm:rank} gives
$\rank_{R_\mathfrak q}(P_\mathfrak q)\ge2^{\height\mathfrak q}=2^{\codim_RN}$.  The rank function of
the projective module $P$ is locally constant, hence constant on the connected spectrum of the
domain $R$, so the localized rank equals $\rank_R(P)$, which proves the bound.
\end{proof}

\begin{remark}\label{rem:domain}
The domain hypothesis only ensures that the rank is well defined: a regular Noetherian ring, being
locally a domain, is a domain exactly when its spectrum is connected, and connectedness is what
makes the rank of a finitely generated projective module a single integer rather than a locally
constant function.
\end{remark}

Carlsson's algebraic formulation concerns dg modules over $S=k[x_1,\dots,x_d]$ with $|x_i|=-1$ and $k$ a field of
characteristic $2$: graded $S$-modules with an $S$-linear differential that is homogeneous of degree
$-1$.  After forgetting the grading these are differential $S$-modules; the grading supplies a free
flag and, for finite-dimensional homology, identifies its support with
$\{\mathfrak m\}$.  Set $\mathfrak m=(x_1,\dots,x_d)$, and call a dg $S$-module \emph{finite free}
when its underlying graded module is finite free.  A finite free dg
$S$-module $P$ is \emph{minimal} if $d_P(P)\subseteq\mathfrak m P$, equivalently if $P\otimes_S k$ has
zero differential.  Homogeneous Gaussian elimination on unit entries decomposes every finite free dg
$S$-module as the direct sum of a minimal module and a contractible module.

\begin{lemma}\label{lem:flag}
Every finite free dg $S$-module admits a free flag; consequently, by Theorem~\ref{cor:TRC}, every
finite free dg $S$-module $P$ with $0<\dim_kH(P)<\infty$ satisfies $\rank_S(P)\ge2^{d}$.
\end{lemma}

\begin{proof}
Homogeneous Gaussian elimination reduces to the minimal summand $Y$: each contractible summand of
rank two removed in the process has a free flag, and free flags on direct summands combine to a free
flag on their direct sum.  Induct on $r=\rank_S(Y)$.  For a homogeneous basis
$b_1,\dots,b_r$, minimality and degree give
$(d_Y)_{ij}\ne0\Rightarrow |b_i|\ge|b_j|$.  Thus, if $W$ is the $k$-span of the basis vectors of
maximal degree and $S\cdot W\subseteq Y$ is the free $S$-submodule they generate, then
$d_Y(S\cdot W)\subseteq S\cdot W$.  Moreover, the entries of the restricted differential lie in
$S_{-1}=\bigoplus_{i=1}^d kx_i$, so
\[
d_Y|_{S\cdot W}=\sum_{i=1}^d x_iL_i,\qquad L_i\in\operatorname{End}_k(W).
\]
Comparing coefficients in $d_Y^2=0$ gives $L_i^2=0$ and, in characteristic $2$,
$L_iL_j=L_jL_i$.  Successively taking kernels shows that these commuting square-zero operators have
a common nonzero kernel; choose $0\ne v\in\bigcap_i\ker L_i$.  Since $W$ is concentrated in one
degree, $S\cdot v$ is a graded direct summand with zero differential, and $Y/(S\cdot v)$ is minimal
and free of rank $r-1$.  By induction, adjoining $S\cdot v$ to the inverse images of a free flag on
$Y/(S\cdot v)$ gives a free flag on $Y$.

If $0<\dim_kH(P)<\infty$, its grading gives $\Supp_SH(P)=\{\mathfrak m\}$, so
$\codim_SH(P)=d$; Theorem~\ref{cor:TRC}, after forgetting the grading, gives
$\rank_S(P)\ge2^d$.
\end{proof}

\begin{proof}[Proof of Corollary~\ref{cor:orbits}]
Set $S=H^*(BE;\Ff_2)=\Ff_2[t_1,\dots,t_d]$, write $X_{hE}$ for the Borel construction of $X$,
and set $H_E^*(X;\Ff_2)=H^*(X_{hE};\Ff_2)$.  Allday and Puppe's construction using singular cochains gives
an $S$-linear differential $D$ on $H^*(X;\Ff_2)\otimes_{\Ff_2}S$ such that the dg $S$-module
\[
 \mathcal H_X=(H^*(X;\Ff_2)\otimes_{\Ff_2}S,D)
\]
has $H^*(\mathcal H_X)\cong H_E^*(X;\Ff_2)$; this is the minimal Hirsch--Brown model, meaning that
$D(\mathcal H_X)\subseteq(t_1,\dots,t_d)\mathcal H_X$.  See
\cite[\S3.11, especially Recollections~(3.11.1)(3)--(4) and the construction preceding
Definition~(3.11.20)]{AlldayPuppe}.  Its underlying $S$-module is free of rank
$\sum_i\dim_{\Ff_2}H^i(X;\Ff_2)=\sum_i\dim_{\Ff_2}H_i(X;\Ff_2)$.
The construction requires neither a cellular action nor triviality of the induced action on
cohomology.
After negating degrees, so that $D$ and the $t_i$ have degree $-1$, $\mathcal H_X$ is a finite free dg
$S$-module with nonzero finitely generated homology.  Lemma~\ref{lem:flag} supplies a free flag on
$\mathcal H_X$, so
Theorem~\ref{cor:TRC} gives
\[
 \sum_i\dim_{\Ff_2}H_i(X;\Ff_2)\ge2^{\codim_S H_E^*(X;\Ff_2)}.
\]
Quillen's dimension theorem~\cite[Part~I, Theorem~7.7]{Quillen} gives
\[
 \dim_S H_E^*(X;\Ff_2)
 =\max\{\rank H:H\le E,\ X^H\ne\varnothing\}=s.
\]
Thus the codimension is $d-s$, and orbit-stabilizer gives
$2^{d-s}=\min_{x\in X}|E\cdot x|$.  When the action is free, $s=0$, giving Carlsson's conjecture.
For the inequality concerning $G$, choose $E\le G$ of rank $r_2(G)$.  Since $E_x=E\cap G_x$, one
has $\max_{x\in X}\rank E_x\le\max_{x\in X}r_2(G_x)$; applying the preceding bound to the
restricted $E$-action gives the stated inequality.
\end{proof}

\begin{proof}[Proof of Corollary~\ref{cor:Carlsson}]
First let $E=(\Z/2)^d$.  Carlsson's Koszul duality identifies the asserted lower bound for $E$ with
the statement that every finite free dg $S$-module with nonzero finite-dimensional homology has rank at
least $2^d$, which is Lemma~\ref{lem:flag}.  More precisely, after splitting off contractible
summands from the $kE$-complex and its corresponding dg $S$-module, the equivalence identifies the total
$k$-dimension of the homology of the complex with the $S$-rank of the dg module, and it preserves
nontriviality and finite-dimensionality of homology.  See
\cite[Theorem~II.7 and Proposition~II.9]{Carlsson}, stated over $\Ff_2$,
and~\cite[Remark~6.2]{VW} for general $k$.

For a finite group $G$, choose $E\le G$ of rank $r_2(G)$.  Since $kG$ is free over $kE$, restriction
carries projective $kG$-modules to projective, hence free, $kE$-modules.  It does not change the
underlying homology vector spaces, so the $E$-case gives the lower bound.

For sharpness, put $r=r_2(G)$.  If $r=0$, then $|G|$ is odd, so Maschke's theorem makes the complex
$k$ concentrated in degree $0$ projective.  Suppose $r>0$.  By finite generation of group
cohomology~\cite[Corollary~6.2]{Evens} and
Quillen's dimension theorem~\cite[Part~I, Theorem~7.7]{Quillen},
$\dim_{\mathrm{Krull}}H^*(G;k)=r$; choose a homogeneous system of
parameters.  The Benson--Carlson parameter construction~\cite[\S4, especially
Theorem~4.1]{BensonCarlson}, written out in~\cite[proof of the main theorem]{CarlsonRanks}, produces a
bounded complex of projective $kG$-modules whose homology, as a graded $k$-vector space, is that of a
product of $r$ spheres, and hence has total dimension $2^r$.  This construction is valid in
characteristic $2$; the hypothesis of odd characteristic in~\cite{CarlsonRanks} enters only the
subsequent mapping cone construction.  Finally, the restriction of a free $G$-action to $E$ is
free, so the topological assertion of Corollary~\ref{cor:Carlsson} follows from
Corollary~\ref{cor:orbits} with $s=0$.
\end{proof}
\section{The filtration spectral sequence}\label{sec:hypertor}

This section proves Theorem~\ref{thm:main} in the slightly more general form of
Theorem~\ref{thm:flag}.
The conventions of Section~\ref{sec:deduce} remain in force.  Fix a differential $A$-module $P$
with a free flag and a finite free resolution $F_\bullet\to H(P)$, with differentials
$\delta_q\colon F_q\to F_{q-1}$.  Lemma~\ref{lem:filtration} replaces $P$, up to
quasi-isomorphism, by a filtered differential module $X$ with $X^q/X^{q-1}\cong F_q$, zero
differential on each subquotient, and $d_X$ inducing $\delta_q$ on adjacent subquotients.  By
Lemmas~\ref{lem:contract} and~\ref{lem:htpy}, the
quasi-isomorphism $X\to P$ induces a homotopy equivalence
$\Tate_A(X)\simeq\Tate_A(P)$.  The exact homology sequences supplied by
Lemma~\ref{lem:htpy} for the filtration of $X$ form a finite exact couple and hence a spectral
sequence converging to $H(\Tate_A(P))$.
Lemma~\ref{lem:zerodiff} identifies the terms on its first page with $A_F\otimes_AF_q$, and the proof of
Theorem~\ref{thm:flag} identifies its first differential with $1\otimes\delta_\bullet$.  Under the
flatness hypothesis of Theorem~\ref{thm:flag}, the sequence therefore collapses at the second page.
Remark~\ref{rem:abc} relates this
filtration spectral sequence to the Adams spectral sequence.

\begin{lemma}\label{lem:filtration}
Let $P$ be a finite projective differential $A$-module whose homology admits a finite resolution
\[
0\to F_m\xrightarrow{\ \delta_m\ }F_{m-1}\xrightarrow{\ \delta_{m-1}\ }\cdots
\xrightarrow{\ \delta_1\ }F_0\to H(P)\to0
\]
by finite free modules.  Then there is a finite free differential $A$-module $X$ quasi-isomorphic to
$P$, with a finite filtration $0=X^{-1}\subseteq X^{0}\subseteq\cdots\subseteq X^{m}=X$ satisfying
$d_X(X^{q})\subseteq X^{q-1}$, together with isomorphisms $X^{q}/X^{q-1}\cong F_q$ for
$0\le q\le m$, under which the map $X^{q}/X^{q-1}\to X^{q-1}/X^{q-2}$ induced by $d_X$ equals
$\delta_q$ for every $q\ge1$.
\end{lemma}

\begin{proof}
The following construction is the finite free specialization of Brown--Erman's construction
\cite[Theorem~3.2]{BrownErman}; see also
\cite[Theorems~2.23 and~2.26]{VandeBogertFlagged}.  We give it explicitly because its filtration is
used in the proof of Theorem~\ref{thm:flag}.
Put $I_0=H(P)$, $I_q=\im\delta_q$ for $1\le q\le m$, and $I_{m+1}=0$.  Starting with $P_0=P$,
construct $P_q$ so that, for $q\ge1$, the projection $\pi_{q-1}\colon P_q\to F_{q-1}$ identifies
$H(P_q)$ with $I_q$.
Let $\alpha_0\colon F_0\to I_0$ be the augmentation and, for $q\ge1$, let
$\alpha_q\colon F_q\to I_q$ be $\delta_q$ with codomain restricted to its image.  Choose a cycle lift
$\varphi_q\colon F_q\to P_q$ of $\alpha_q$, with $\pi_{q-1}\varphi_q=\delta_q$ for $q\ge1$, and set
\[
 P_{q+1}=P_q\oplus F_q,\qquad d_{P_{q+1}}(x,f)=(d_{P_q}x+\varphi_qf,\,0).
\]
The exact homology sequence for $0\to P_q\to P_{q+1}\to F_q\to0$ identifies $H(P_{q+1})$, via the
projection $\pi_q\colon P_{q+1}\to F_q$, with $\ker\alpha_q=I_{q+1}$.  Thus $P_{m+1}$ is acyclic.
Unwinding,
$P_{m+1}=P\oplus F_0\oplus\cdots\oplus F_m$ as a module, with
differential sending $x\in P$ to $d_Px$ and $f\in F_q$ to
$\varphi_qf\in P\oplus F_0\oplus\cdots\oplus F_{q-1}$.  The quotient $X=P_{m+1}/P$, filtered by
$X^{q}=F_0\oplus\cdots\oplus F_q$, is therefore finite free with $d_X(X^{q})\subseteq X^{q-1}$ and with
$\pi_{q-1}\varphi_q=\delta_q$ induced on subquotients.  Finally, the map $\psi\colon X\to P$
collecting the $P$-components of the $\varphi_q$ is a morphism, because the differential of
$P_{m+1}$ squares to zero, and $P_{m+1}$ is its mapping cone.  Since $H(P_{m+1})=0$, the exact
homology sequence for $0\to P\to P_{m+1}\to X\to0$ shows that $\psi$ is a quasi-isomorphism.
\end{proof}

\begin{remark}
For a bounded complex, the filtered module of Lemma~\ref{lem:filtration} may be obtained by
totalizing the Cartan--Eilenberg construction of~\cite[\S5.7]{Weibel} and
\cite[Example~2.27]{VandeBogertFlagged} by direct sums and then forgetting the grading.  In the
homotopy category, this filtration and its subquotients constitute an Adams resolution in the sense
of Gugenheim--May~\cite{GugenheimMay} and
\cite[Definition~10.2, Theorem~11.1]{BarthelMayRiehl}.
\end{remark}

\begin{lemma}\label{lem:zerodiff}
Let $V$ be a finite free $A$-module, regarded as a differential module with zero differential.  There
is a natural isomorphism
\[
\gamma_V\colon A_F\otimes_AV\ \xrightarrow{\ \cong\ }\ H\bigl(\Tate_A(V)\bigr),\qquad
\gamma_V(1\otimes v)=[v\otimes v].
\]
\end{lemma}

\begin{proof}
Since $V$ has zero differential, $\Tate_A(V)=(V\otimes V,1+\tau)$, and hence
$H(\Tate_A(V))=N(V)$ with $N(V)=\ker(1+\tau)/\im(1+\tau)$.  The assignment $v\mapsto[v\otimes v]$
is additive into $N(V)$, because $(v+w)\otimes(v+w)+v\otimes v+w\otimes w=(1+\tau)(v\otimes w)$, and
Frobenius semilinear, because $(av)\otimes(av)=a^{2}(v\otimes v)$.  It therefore induces the
$A$-linear map $\gamma_V$.  For a basis $e_1,\dots,e_r$ of $V$, the orbit sums
$e_i\otimes e_j+e_j\otimes e_i$ $(i<j)$ span $\im(1+\tau)$, while the diagonal tensors represent the
remaining classes in $\ker(1+\tau)$.  Thus the classes $[e_i\otimes e_i]$ form a basis of $N(V)$.
They are the images of the basis $1\otimes e_i$ of $A_F\otimes V$, so $\gamma_V$ is an isomorphism,
natural in $V$ by construction.
\end{proof}

\begin{remark}
The isomorphism of Lemma~\ref{lem:zerodiff} does not arise from a diagonal map
$A_F\otimes V\to V\otimes V$: the assignment $v\mapsto v\otimes v$ is Frobenius semilinear in the
scalars but not additive, so it induces no $A$-linear map $A_F\otimes V\to V\otimes V$; additivity
appears only after passing to the quotient $N(V)$.
\end{remark}

\begin{theorem}\label{thm:flag}
Let $A$ be a commutative Noetherian ring with $2=0$ whose Frobenius bimodule $A_F$ is flat as a
right $A$-module, and let $P$ be a differential $A$-module with a free flag whose homology admits a
finite resolution by finite free modules.  Then
\[
H\bigl(\Tate_A(P)\bigr)\ \cong\ A_F\otimes_AH(P).
\]
\end{theorem}

\begin{proof}
The differential module $X$ of Lemma~\ref{lem:filtration} is quasi-isomorphic to $P$, and both admit free
flags; the quasi-isomorphism between them is therefore a homotopy equivalence
(Lemma~\ref{lem:contract}), and $\Tate_A$ preserves homotopy equivalences
(Lemma~\ref{lem:htpy}).  Replacing $P$ by $X$, we may assume that $P$ carries the filtration
constructed in Lemma~\ref{lem:filtration}, namely
$0=P^{-1}\subseteq P^{0}\subseteq\cdots\subseteq P^{m}=P$ with subquotients
$P^{q}/P^{q-1}=F_{q}$ and induced maps $\delta_{q}$.

By Lemma~\ref{lem:htpy}, applying
$\Tate_A$ to the short exact sequences $0\to P^{q-1}\to P^{q}\to F_{q}\to0$ gives exact homology
sequences.  These form a finite exact couple and hence a spectral sequence of $A$-modules, singly
graded by the filtration degree $q$, converging to $H(\Tate_A(P))$.  Its first page is
$E^1_{q}=H(\Tate_A(F_{q}))$ with differential
$d^1\colon E^1_{q}\to E^1_{q-1}$ induced by the connecting maps, and Lemma~\ref{lem:zerodiff} gives
$E^1_{q}\cong A_F\otimes F_q$, naturally in $F_{q}$.

We claim that, under the isomorphisms $\gamma_{F_q}$ of Lemma~\ref{lem:zerodiff},
$d^1=A_F\otimes_A\delta_q$.  Write $D=\partial+(1+\tau)$.  If
$\tilde v\in P^{q}$ lifts $v\in F_q$, then the cross terms cancel because $2=0$ and
\[
D\bigl(\tilde v\otimes\tilde v+d_P\tilde v\otimes\tilde v\bigr)
 =d_P\tilde v\otimes d_P\tilde v.
\]
Thus the element on the left is a cycle modulo $\Tate_A(P^{q-1})$ representing
$[v\otimes v]$, and the right side is a cycle in $\Tate_A(P^{q-1})$ whose image in
$\Tate_A(F_{q-1})$ is $\delta_qv\otimes\delta_qv$.  Under the identifications $\gamma$ of
Lemma~\ref{lem:zerodiff}, the differential is therefore $d^1=A_F\otimes\delta_q$.  Hence
$(E^1_{\bullet},d^1)$ is exactly $A_F\otimes_A-$ applied to the chosen resolution:
\[
\cdots\longrightarrow A_F\otimes F_q\xrightarrow{\ 1\otimes\delta_q\ }A_F\otimes F_{q-1}
\longrightarrow\cdots\longrightarrow A_F\otimes F_0.
\]
Thus $E^2_{q}\cong\Tor^A_q(A_F,H(P))$.  Flatness of $A_F$ makes these vanish for $q>0$,
while $E^2_{0}=A_F\otimes H(P)$.  Since the second page is concentrated in filtration degree $0$,
the sequence collapses and yields the isomorphism in the theorem.
\end{proof}

\begin{proof}[Proof of Theorem~\ref{thm:main}]
Since $R$ is regular local, the finitely generated module $H(P)$ admits a finite free resolution,
and the Frobenius of $R$ is flat by Kunz's theorem; Theorem~\ref{thm:flag} gives
$H(\Tate_R(P))\cong\Rone\otimes_RH(P)$, and flatness identifies $\Rone\otimes_RH(P)$ with
$H(\Frob P)$.  When $H(P)$ has finite length, Lemma~\ref{lem:scaling} gives
$\htot(\Tate_R(P))=2^{d}\,\htot(P)$.
\end{proof}

\begin{remark}\label{rem:abc}
The spectral sequence of Theorem~\ref{thm:flag} is the Adams spectral sequence for the projective
class generated by finite free modules with zero differential; its \emph{ghosts} are the morphisms
that induce zero on homology.  In the triangulated formalism, this is an Adams spectral sequence
associated to a projective class; see Adams~\cite{Adams}, Brinkmann~\cite{Brinkmann},
Christensen~\cite{Christensen}, and Meyer~\cite{Meyer}.  Christensen
\cite[\S8, Proposition~8.2]{Christensen} identifies the ghost
projective class, while Meyer~\cite[Theorem~4.3]{Meyer} identifies the second page with the left
derived functors, here $\Tor^A_q(A_F,H(P))$.  The finite filtration of
Lemma~\ref{lem:filtration} shows that $P$ belongs to the thick subcategory generated by finite free
modules with zero differential, so convergence follows from
Meyer~\cite[Proposition~4.5]{Meyer}.  Over an abelian category the sequence specializes to
Grothendieck's hyperhomology spectral sequence~\cite[\S6]{Meyer}; the finite filtration used here
makes both the sequence and its convergence explicit.
\end{remark}

\begin{remark}
Flatness of $A_F$ enters the proof of Theorem~\ref{thm:flag} only through the vanishing of
$\Tor^A_q(A_F,H(P))$ for $q>0$.  Over a Noetherian local ring $R$ of characteristic $2$, if the
homology $H(P)$ has finite projective dimension, then the resolution required by
Lemma~\ref{lem:filtration} exists and Peskine--Szpiro's acyclicity theorem for the
Frobenius~\cite[Chapter~I, Theorem~(1.7)]{PeskineSzpiro} gives this vanishing, so the spectral
sequence of Theorem~\ref{thm:flag} computes
$H(\Tate_R(P))\cong\Rone\otimes_RH(P)$.  Thus, in this argument, regularity is used only to ensure
that $H(P)$ has finite projective dimension.  This extension beyond regular rings is not needed for
Theorem~\ref{thm:rank}.
\end{remark}

\begin{example}
Regularity, equivalently flatness of the Frobenius, is essential to Theorem~\ref{thm:main}: over a
singular ring one can have $\htot(\Tate_R(P))<\htot(F_R^*P)$.  Let
\[
R=\Ff_2[a_1,a_2,a_3,b_1,b_2,b_3]\big/\bigl(a_1b_1+a_2b_2+a_3b_3,\ a_i^2,\ b_i^2\bigr),
\]
the exterior algebra $\bigwedge V$ of an even-dimensional $\Ff_2$-vector space $V$, modulo its standard
symplectic form $\omega=\sum_i a_i\wedge b_i$; here $\dim V=6$ and $\length_R R=36$.
The defining relation $\sum_{i=1}^3 a_ib_i=0$ makes
$P=R\xleftarrow{(a_i)}R^{3}\xleftarrow{(b_i)^{\mathsf T}}R$ a complex.  Every
variable squares to zero, so $\Frob P$ has zero differential, and a computation in
Macaulay2~\cite{M2} gives
\[
\htot(\Frob P)=180,\qquad \htot(P\otimes_R P)=266,\qquad \htot(\Tate_R(P))=178.
\]
Thus $\htot(\Tate_R(P))=178\ne180=\htot(\Frob P)$, so the conclusion of Theorem~\ref{thm:main} fails
over this singular ring; here $H(P)$ has infinite projective dimension.  The inequality
$\htot(\Frob P)\le\htot(P\otimes_R P)$ nonetheless holds, $180\le266$.  For
\[
 R'=\Ff_2[a_1,\ldots,a_4,b_1,\ldots,b_4]/\bigl(\textstyle\sum_{i=1}^4a_ib_i,\ a_i^2,\ b_i^2\bigr),
 \qquad P'=R'\xleftarrow{(a_i)}(R')^4\xleftarrow{(b_i)^{\mathsf T}}R',
\]
the computation gives $\htot(\Tate_{R'}(P'))=796\ne816=\htot(F_{R'}^*P')$.
\end{example}
\raggedbottom

\end{document}